\documentclass[11pt]{article}
\usepackage{amsmath, amssymb, amsfonts, amsthm,latexsym, hyperref,  mathdots}
\usepackage[all,cmtip]{xy}
\usepackage{graphicx}

\newtheorem*{theorem*}{Theorem}

\newtheorem{theorema}{Theorem}

\newtheorem{lema}[theorema]{Lemma}

\newtheorem{theorem}{Theorem}

\newtheorem{propos}[theorem]{Proposition}

\newtheorem{cor}[theorem]{Corollary}
\newtheorem*{cor*}{Corollary}

\newcommand{\ind}{1{\hskip -2.5 pt}\hbox{I}}

\newcommand{\R}{{\mathbb R}}

\def\Z{\mathbb Z}

\def\N{\mathbb N}

\renewcommand\P{\mathbb P}

\def\ds{\displaystyle}

\title{A note on general sliding window processes}
\author{Noga Alon \thanks{Sackler School of Mathematics and
Blavatnik School of Computer Science,
Tel Aviv University,
Tel Aviv 69978, Israel. Email: nogaa@tau.ac.il.
Research supported in part by a
USA-Israeli
BSF grant, by an ISF grant, by the Hermann Minkowski Minerva Center
for
Geometry at Tel Aviv University and by the Israeli I-Core program.}
\and Ohad N. Feldheim \thanks{Department of Mathematics, Weizmann Institute of Science,
 Rehovot 76100, Israel.  Email: ohad\_f@netvision.net.il.
Partially supported by a grant from the Israel Science Foundation.
} }

\begin{document}
\maketitle
\begin{abstract}
Let $f:\R^k\to \R$ be a measurable function, and let
$\{U_i\}_{i\in\N}$ be a sequence of i.i.d. random variables.
Consider the random process $Z_i=f(U_{i},...,U_{i+k-1})$. We show
that for all $\ell$, there is a positive probability, uniform in
$f$, for $Z_1,...,Z_\ell$ to be monotone. We give upper and lower
bounds for this probability, and draw corollaries for $k$-block
factor processes with a finite range.

The proof is based on an application of combinatorial results from
Ramsey theory to the realm of continuous probability.
\end{abstract}
\section{Introduction}\label{sec: intro}

The objective of this note is to bring to the attention of probabilists an
application of tools from Ramsey theory to probabilistic questions on sliding
window processes. These results could be further extended with relative ease
to other noise type variables. Let $f:\R^n\to \R$ be a measurable function,
and let $\{U_i\}_{i\in\N}$ be a sequence of i.i.d. random variables. Consider
the random process $Z^f_i=f(U_{i},...,U_{i+k-1})$. Such processes are called
\emph{$k$-block factors}.

Our main observation is the following:
\begin{theorem}\label{thm: main cont}
For every $k,\ell\in\N$ there exists $p=p_{k,\ell}>0$ such that for
every measurable $f:\R^k\to \R$, one of the following holds:
\begin{align*}
& & & & & \text{either }& \P(Z^f_1< Z^f_2< \dots < Z^f_{\ell})>p, & & \\
& & & & & \text{or }&\P(Z^f_1= Z^f_2= \dots = Z^f_{\ell})>p, & & \\
& & & & & \text{or }&\P(Z^f_1> Z^f_2> \dots > Z^f_{\ell})>p. & & \\
\end{align*}
\end{theorem}

This theorem is a corollary of the following discrete counterpart:
\begin{theorem}\label{thm: main disc}
For every $k,\ell,r\in\N$ there exists $p=p_{k,\ell,r}>0$ such that
for every measurable $f:\R^k\to \{1,\dots,r\}$ the following holds:
$$\P(Z^f_1 = Z^f_2 = \dots = Z^f_{\ell})>p$$
\end{theorem}
In other words, the probability of every $k$-factor to be constant
on a discrete interval of length $\ell$ is bounded away from zero.

\begin{proof}[Proof of Theorem~\ref{thm: main cont} using Theorem~\ref{thm: main
disc}] Let $k,\ell\in\N$, and let $f:\R^k\to \R$. We define a new
function $g:\R^{k+1}\to\{-1,0,1\}$ by
\begin{equation*}
g(x_1,...,x_{k+1})=
\begin{cases}-1 & f(x_1,...,x_k)>f(x_2,...,x_{k+1})\\
0 & f(x_1,...,x_k)=f(x_2,...,x_{k+1})\\
1 & f(x_1,...,x_k)<f(x_2,...,x_{k+1})\\
\end{cases}
\end{equation*}
By Theorem~\ref{thm: main disc} there exists a positive $p=p_{k+1,\ell-1,3}$
such that one of the following holds:
\begin{align*}
& & & & & \text{either }& \P(Z^g_1 = Z^g_2 = \dots = Z^g_{\ell-1}=-1)>p, & & \\
& & & & & \text{or }& \P(Z^g_1 = Z^g_2 = \dots = Z^g_{\ell-1}=0)>p, & & \\
& & & & & \text{or }& \P(Z^g_1 = Z^g_2 = \dots = Z^g_{\ell-1}=1)>p. & & \\
\end{align*}
The theorem follows.
\end{proof}

The particular case $r=2$ which motivated our interest in the problem, is
presented in the following corollary.

\begin{cor}
For every $k,\ell \in\N$ there exists $p=p_{k,\ell}>0$ such that for
every measurable $f:\R^k\to\{0,1\}$ the following holds:
$$\P(Z^f_1 = Z^f_2 = \dots = Z^f_{\ell})>p$$
\end{cor}

The decay of $p$ as a function of $k$ in the above theorems is very fast. In
particular the $p$ which Theorem~\ref{thm: main disc} yields is $\frac
1{M^{k+l-1}}$ for $M$ satisfying
$$M={\underbrace{2^{2^{{\iddots}^{2}}}}_{\text{$k-2$ times}}}^{^{\hspace{-10pt}\ell^r}}.$$
Such tower dependency is in fact essential as the following proposition
shows:

\begin{theorem}\label{thm: tower is a must}
For all large enough $k$, there exists $f:[0,1]^k\to\{0,1\}$ such that
for the
process $\{Z^f_i\}$ defined as above with respect to uniform variables
$U_i$ on $[0,1]$ the following holds:
$$\P(Z^f_1 = Z^f_2 = \dots = Z^f_{2k})<\frac{ 9 k^2}{M},$$
where $$M =
{\underbrace{2^{2^{{\iddots}^{2}}}}_{\underset{\text{times}}{k-2}}}^{^{k/\sqrt8}}.$$
\end{theorem}

\section{Background and motivation}
The research of $k$-block factors originated as a part of a wider attempt to
understand \emph{$m$-dependent }processes. These generalize independent
processes in discrete time, by requiring that every two events which are
separated by a time-interval with length more than $m$ will be independent.
Such processes arise naturally as scaling limits in renormalization theory
(see for example \cite{OBR}). Clearly, every $k$-block factor is
$(k-1)$-dependent. For a while the converse was also conjectured to hold, to
the extent that in certain papers, results on $k$-block factors are presented
as results on $(k-1)$-dependent processes, conditioned on the validity of the
conjecture (see for example \cite{SJ}).

While for Gaussian processes, every $m$-dependent process is indeed
an $m+1$-block factor, we now know that for general $m$-dependent
processes this is not true. Ibragimov and Linnik have already stated
in 1971 that there should exist a $1$-dependent process which is not a
$2$-block factor, but provided no example. The first example was
published by Aaronson and Gilat in \cite{AG} in 1987. Later, in
\cite{BGRM}, Burton, Goulet and Meester showed that there exists a
$1$-dependent process which is not a $k$-block factor for any $k$.

One property of binary block factors, i.e., block factors with range
$\{0,1\}$, which have been extensively studied, is the probability
of observing $r$ consecutive occurrences of the value $b$ in the
process. This event is called an $r$-\emph{run} of $b$-s. Janson, in
\cite{SJ}, studied the convergence of the statistics of runs of
zeros in a $k$-factor in which every two ones are guaranteed to be
separated by $k-1$ zeros. De Valk, in \cite{V1}, computed the
minimal and maximal possible probability of a $2$-run of ones given
the marginal probability of seeing the value one. Such studies give
rise to the following natural question: is it possible to create a
binary $k$-block factor for some $k$ which almost surely has neither
an $r$-run of zeros nor an $r$-run of ones? Here we show that this
is impossible. The result is twofold. On one hand, the probability
of seeing an arbitrarily long run is bounded away from zero. On the
other hand, it can be extremely small.

\section{Proof of the results}\label{sec: Proof}
This section is dedicated to the proofs of Theorems~\ref{thm: main
disc} and~\ref{thm: tower is a must}. For this purpose we shall use
a classical result on de-Bruijn graphs, whose proof we present for
completion.

For a directed graph $G$ let $\chi(G)$, the chromatic number of $G$,
denote the minimal number of colors required to color the vertices
of $G$ so that no two adjacent vertices get the same
color.

Define $D(k,m)$, the increasing $k$-dimensional de-Bruijn graph of $m$
symbols, to be the directed graph whose vertices are all the strictly
increasing sequences of length $k$ with elements in $\{1,\dots,m\}$, such
that there is a directed edge from the sequence $\{a_1,...,a_k\}$ to the
sequence $\{b_1,...,b_k\}$ if and only if $b_i=a_{i+1}$ for all
$i\in\{1,\dots,k-1\}$.

We shall make use of the fact that $D(k+1,m)$ is the directed line-graph of
$D(k,m)$. That is - that the map
$\phi:(a_1,...,a_{k+1})\to\left((a_1,...,a_{k}),(a_2,...,a_{k+1})\right)$ is
a bijection, mapping every vertex of $D(k+1,m)$ to an edge of $D(k,m)$.

\begin{theorema}\label{thm: debruijn}
$\log_2 \chi\left(D(k,m)\right) \leq \chi \left(D(k+1,m)\right)$.
\end{theorema}

\begin{proof}
Using the fact that $D(k+1,m)$ is the directed line-graph of $D(k,m)$, we get
that a vertex coloring of $D(k+1,m)$ is equivalent to an edge coloring of
$D(k,m)$ in which there is no monochromatic directed path of length $2$.
Thus, it is enough to show that for every such coloring of $E(D(k,m))$ using
$q$ colors, there exists a coloring of $V(D(k,m))$ using $2^q$ colors.

Let $C:E(D(k,m))\rightarrow\{1,\dots,q\}$ be an edge-coloring of $D(k,m)$ as
above. Construct $C':V(D(k,m))\rightarrow\mathcal{P}\{1,\dots,q\}$ using the
subsets of $\{1,\dots,q\}$ as colors in the following way. Define
$C'(u)=\{C(u,v)\ :\ (u,v)\in E(D(k,m))\}$. To see that $C'$ is a proper
vertex coloring, observe that if $C'(u)=C'(v)$ and $(u,v)\in E(D(k,m))$ then
$C(u,v)\in C'(v)$ which implies the existence of $(v,w)$ such that
$C(v,w)=C(u,v)$, in contradiction to our premises.
\end{proof}

Since clearly $\chi(D(1,m))=m$, we get that for $k\ge2$,
\begin{equation}\label{eq: bound}
\chi(D(k,m)) \geq \log_2^{(k-1)}(m),
\end{equation}
where $\log_2^{(k)}$ represents $k$ iterations of the the function $\log_2$.

We now use the following theorem by Chv\'atal \cite{CV}.
\begin{theorema}[Chv\'atal]\label{thm: chvatal}
Let $D$ be a directed graph and let $\ell, r\in\N$ such that
$\chi(D)>\ell^r$; then any edge-coloring of $D$ with $r$ colors contains a
monochromatic directed path of $\ell$ edges.
\end{theorema}

Combining this with the fact that $D(k+1,m)$ is the directed line-graph of
$D(k,m)$ and with \eqref{eq: bound}, we draw the following corollary.

\begin{cor}\label{cor: maincor}
Given $k,\ell,r\in\N$, let $M$ be an integer such that
$\log_2^{(k-2)}(M)={\ell}^r$.
 Then any $r$-coloring of the vertices of
$D(k,M)$ contains a monochromatic directed path of $\ell$ vertices.
\end{cor}
Using this we are ready to prove Theorem~\ref{thm: main disc}.

\begin{proof}[Proof of Theorem~\ref{thm: main disc}] Let $r\in \N$, let $f:[0,1]^k\to\{1,\dots, r\}$ be a
measurable function, $U_i$ uniform on $[0,1]$ and
$\{Z^f_i=f(U_{i},...,U_{i+k-1})\}_{i\in\Z}$. Observe that since
measurable sets on $\R$ and on $[0,1]$ are isomorphic, our choice of
the distribution of $U_i$ does not limit the generality of our
proof. Choose $M=M(k,\ell,r)$ as in Corollary~\ref{cor: maincor} to
get:

\begin{align*}\label{eq: comp}
\ds \P & \left(X_1=\dots =X_\ell \right)=\int_0^1 dx_1 \cdots
\int_0^1 dx_{l+k-1} \ \ind \left\{
f(x_1,\dots,x_k)=\dots=f(x_l,\dots,x_{l+k-1})\right\}\\
& = \int_0^1 dy_\ell\cdots \int_0^1 dy_M  \ \frac {(M-k-l+1)!} {M!}
\sum_{1\le j_1<\dots<j_{l+k-1}\le M} \ind \left\{
f(y_{j_1},\dots,y_{j_k})=\dots=f(y_{j_l},\dots,y_{j_{l+k-1}})\right\}.
\end{align*}
where the equality in the second line is obtained by first picking
$M$ random i.i.d.
values in $[0,1]$ and then by assigning a random set of $l+k-1$ of
them to the
variables $x_1,\dots,x_{l+k-1}$ uniformly at random.

Now, for a given $\bar y=(y_1,\dots,y_M)\in[0,1]^M$, the inner sum
counts the number of monochromatic directed paths in $D(k,M)$, when
coloring by
$$c(a_1,\dots, a_{k}) = f(y_{a_1},\dots,y_{a_k}).$$
This is an $r$-coloring, therefore by the above corollary, this
inner sum is at least $1$. We conclude that
$$\P \left(X_1=\dots =X_l \right) \ge \frac {(M-k-l+1)!} {M!}>\frac{1}{M^{k+l-1}},$$
as required.
\end{proof}

\subsection{Tower dependency is essential}

This subsection contains the proof of Theorem~\ref{thm: tower is a
must}.

For $i>1$, the 2-tower function, $t_i$, denotes the function
satisfying $t_i(k)=2^{t_{i-1}(k)}$, and $t_1(k)=k$. Also, recall the
notation $D(k,m)$ of the increasing $k$-dimensional de-Bruijn graph
of $m$ symbols which is defined in the beginning of this section.

In our proof we use the following lemma of Moshkovitz and Shapira (see
\cite[Corallary~3]{MS}).

\begin{lema}\label{Lema: MS}
There exists $n_0\in \N$ such that for any $k\ge 3$, $q\ge 2$ and $n>n_0,$
there exists an edge coloring of $D(k,t_{k-1}(n^{q-1}/\sqrt8))$ with $q$
colors which contains no monochromatic path of length $n$.
\end{lema}

Recalling that edge colorings of $D(k-1,m)$ are the same as vertex colorings
of $D(k,m)$, and plugging $q=2$, $n=k$ in Lemma~\ref{Lema: MS}, we get the
following useful proposition.

\begin{propos}\label{Prop: MS}
For every large enough $k$, there exists a vertex $2$-coloring of
$D(k,t_{k-2}(k/\sqrt8))$ such that no path of length $k$ is
monochromatic.
\end{propos}

We are now ready to prove Theorem~\ref{thm: tower is a must}.
\begin{proof}[Proof of Theorem~\ref{thm: tower is a must}]
Let $M=t_{k-2}(k/\sqrt8)$, and let $g$ be a vertex $2$-coloring of
$D(k,M)$ by the colors $\{0,1\}$
such that no path of length $k$ is monochromatic, as exists
by Proposition~\ref{Prop: MS}. Define $h:\{1,\dots,M\}^k\to\{0,1\}$
as follows:
\begin{equation}\label{eq: four cases}
h(z_1,\dots,z_k)=\begin{cases}
g(z_1,\dots,z_k) & z_1<\dots<z_k\\
g(z_k,\dots,z_1) & z_1>\dots>z_k\\
0 & \exists i\neq j\text{ s.t. }z_i=z_j\\
\alpha(z_2,z_3)  & \text{otherwise},
\end{cases}
\end{equation}
where $\alpha(x,y)$ takes the value $0$ if $x<y$, and $1$ otherwise.

 Let $z_1,\dots,z_{3k}$ be distinct integers in $\{1,\dots,M\}$. We
claim that the following is impossible:
\begin{equation}\label{eq: h const}
h(z_1,\dots,z_{k})=h(z_2,\dots,z_{k+1})=\dots=h(z_{2k+1},\dots,z_{3k}).
\end{equation}

Assuming the contrary, we study two cases.

The first case is when $z_2,\dots,z_{2k}$ is monotone. In this case
$h$ is equal to $g$ along the path
$$(z_{2},\dots,z_{k+1}), (z_{3},\dots,z_{k+2}), \dots,
(z_{k+1},\dots,z_{2k})$$ (by the first or second case of the definition
\eqref{eq: four cases}). This is a contradiction, since $g$ cannot be
constant along a path of length $k$ in $D(k,M)$.

In the complimentary case there exists a local extremum among
$z_3,\dots,z_{2k-1}$, i.e., there exists $i\in \{3,\dots,2k-1\}$
such that either $z_i> \max\{z_{i-1},z_{i+1}\}$ or
$z_i<\min\{z_{i-1},z_{i+1}\}$. Thus, the values
\begin{align*}
&h(z_{i-2},z_{i-1},\dots,z_{i+k-3}) = \alpha(z_{i-1},z_{i}), \: \text{and}\\
& h(z_{i-1},z_i, \dots,z_{i+k-2}) = \alpha(z_i,z_{i+1})
\end{align*}
are not equal, which also leads to a contradiction.

Now, observe that taking uniform distribution over $1,\dots, M$ the
probability that $(z_i)_{i\in\{1,\dots,3k\}}$ are distinct is
greater than
$$\prod_{j=0}^{3k-1} \left(1-\frac j M\right)>1-\frac{9k^2}{M}.$$
We may therefore define $f:[0,1]^k\to \{0,1\}$ to be
$f(x_1,\dots,x_k)=h(\lceil M x_1 \rceil,\dots,\lceil M x_k \rceil)$
and get
$$\P(Z^f_1 = Z^f_2 = \dots = Z^f_{2k})<\frac{9k^2}{M}$$ as
required.
\end{proof}

\section{Acknowledgments}\label{sec: remarks}

We wish to thank P\'{e}ter Mester and Boris Tsirelson for useful
discussions, which ignited and fanned our interest in the problem.

\end{document}